\documentclass[a4paper,11pt]{amsart}

\usepackage{style}
\usepackage{graphicx}
\usepackage{eucal}
\usepackage{amssymb}
\usepackage{mathrsfs}
\usepackage{amssymb}
\usepackage{amsmath}
\usepackage{picture}
\usepackage[margin=3cm]{geometry}
\usepackage{lipsum}
\usepackage{graphics}
\usepackage{mathrsfs}
\usepackage{amsmath,amssymb,amsthm,amscd,amsfonts} 
\usepackage{mathrsfs}
\usepackage{fancyhdr}
\usepackage{colonequals}
\usepackage[all]{xy}
\usepackage{stmaryrd}
\usepackage{color}
\usepackage{multicol}
\usepackage{mathtools}
\usepackage{hyperref}
\usepackage{colonequals}
\usepackage{tikz-cd}
\usepackage{dcpic,pictex}
\usepackage{algorithm2e}  
\usepackage{graphicx}
\usepackage{eucal}
\usepackage{graphics}
\usepackage{color,cite,graphicx}
\usepackage{tikz}
\definecolor{gray}{gray}{0.4}
\usepackage[all]{xy}
\usepackage{rotating}
\usepackage{caption}
\usepackage{yfonts}
\usepackage{multicol}
\usepackage{ulem}
\DeclareMathOperator{\Ktop}{K_{top}}
\DeclareMathOperator{\pr}{pr}
\DeclareMathOperator{\Sym}{Sym}

\DeclareMathOperator{\og}{OG10}
\DeclareMathOperator{\Hom}{Hom}

\DeclareMathOperator{\Sign}{sign}
\DeclareMathOperator{\Stab}{Stab}

\DeclareMathOperator{\NS}{NS}
\DeclareMathOperator{\T}{T}

\DeclareMathOperator{\rank}{rk}

\DeclareMathOperator{\rk}{rk}
\DeclareMathOperator{\Aut}{Aut}
\DeclareMathOperator{\Bir}{Bir}

\DeclareMathOperator{\Mon}{Mon}
\DeclareMathOperator{\Br}{Br}
\DeclareMathOperator{\prim}{prim}
\DeclareMathOperator{\bL}{\mathbf{L}}

\DeclareMathOperator{\A}{\mathbf{A}}
\DeclareMathOperator{\D}{\mathbf{D}}
\DeclareMathOperator{\E}{\mathbf{E}}

\DeclareMathOperator{\U}{\mathbf{U}}

\DeclareMathOperator{\bLambda}{\boldsymbol{\Lambda}}

\newcommand{\msv}{M_v(S,\theta)}
\newcommand{\msvs}{{M}^s_v(S,\theta)}
\newcommand{\tmsv}{\widetilde{M}_v(S,\theta)}
\newcommand{\tmsvtwisted}{\widetilde{M}_v(S,\alpha,\theta)}
\newcommand{\msvtwisted}{M_v(S,\alpha,\theta)}
\title[O'Grady tenfolds as moduli spaces of sheaves]{O'Grady tenfolds as moduli spaces of sheaves}

\author[Camilla Felisetti]{Camilla Felisetti}
\address[Camilla Felisetti]{Dipartimento di Scienze Fisiche, Informatiche e Matematiche, Università degli Studi di Modena e Reggio Emilia, Via Campi 213/B, 41125 Modena, Italy }
\email{camilla.felisetti@unimore.it}

\author{Franco Giovenzana}
\address[Franco Giovenzana]{Fakultät für Mathematik, TU Chemnitz, Reichenhainer Str. 39, 09126 Chemnitz, Germany}
\curraddr{Laboratoire de mathématiques d’Orsay, Université Paris Saclay, Rue Michel Magat, Bât. 307, 91405 Orsay, France}
\email{franco.giovenzana@universite-paris-saclay.fr}

\author{Annalisa Grossi}
\address[Annalisa Grossi]{Fakultät für Mathematik, TU Chemnitz, Reichenhainer Str. 39, 09126 Chemnitz, Germany}
\curraddr{Laboratoire de mathématiques d’Orsay, Université Paris Saclay, Rue Michel Magat, Bât. 307, 91405 Orsay, France}
\email{annalisa.grossi@universite-paris-saclay.fr}

\makeatletter
\@namedef{subjclassname@2020}{%
	\textup{2020} Mathematics Subject Classification}
\makeatother

\usepackage{microtype}

\begin{document}

\maketitle
\begin{abstract}	
We give a lattice-theoretic characterization for a manifold of \(\og\) type to be birational to some moduli space of (twisted) sheaves on a K3 surface. We apply it to the Li--Pertusi--Zhao variety of \(\og\) type associated to any smooth cubic fourfold. Moreover we determine when a birational transformation is induced by an automorphism of the K3 surface and we use this to classify all induced birational symplectic involutions.

\end{abstract}

\section{Introduction}
 Knowing birational models of irreducible holomorphic symplectic (ihs) manifolds is a significant step towards the full comprehension of their geometry.
 In recent years, the birational geometry of ihs manifolds and their deformation theory have had extensive applications in many fields of algebraic geometry: not only they are a fundamental tool for hunting new examples of ihs varieties in both the smooth and singular case (see for example \cite{Menet22,BGMM24}) but also they turned out to be one of the key ingredients in the investigation of P=W phenomena arising from non abelian Hodge theory, see for example \cite{dCMS,SY,FM,FSY}.
 
 Many examples of ihs manifolds are realized from moduli spaces of sheaves on abelian or K3 surfaces. 
 This is the case of the two O'Grady examples $\mathrm{OG}6$ and $\og$ of dimensions 6 and 10, which are constructed as symplectic resolutions of these moduli spaces. 
 In this paper we focus on ihs manifolds deformation equivalent to $\og$ and throughout the dissertation we refer to them as manifolds of $\og$ type.
 Since any two birational ihs manifolds are deformation equivalent but the converse is far from being true (see \cite[Theorem 2.5]{H03}), it is natural to investigate when an ihs manifold in this deformation class is birational to a symplectic resolution of a moduli space of sheaves on a K3 surface.
 The first motivation for this paper was to provide an answer to this question.
 As it is well known, the geometry of an ihs manifold $X$ is encoded in its second integral cohomology group $H^2(X,\mathbb{Z})$. This group is torsion free and carries a symmetric non-degenerate bilinear form, called the Beauville–Bogomolov–Fujiki form, which induces a lattice structure on it.
 The first main result of the paper provides an answer to this question via a lattice-theoretic characterization. To this end, in accordance with the already existing terminology \cite{MW15,Grossi:induced} we introduce the lattice-theoretic notion of numerical moduli space (see \autoref{nms}).
\begin{theorem*}[\ref{FGG}]
Let $(X,\eta)$ be a marked pair of $\og$ type. The following conditions are equivalent:
\begin{enumerate}[(i)]
    \item There exists a K3 surface $S$, a Mukai vector $v=2w$ where $w$ is a primitive Mukai vector of square $2$, and a $v-$generic polarization $\theta$ on $S$ such that $X$ is birational to $\tmsv$;
    \item The manifold \(X\) is a numerical moduli space.
\end{enumerate}
\end{theorem*}

While this extends analogous results for the other known deformation types of ihs manifolds \cite{MW15,Grossi:induced}, the theory of moduli spaces of sheaves on K3 surfaces developed into considering also their twisted analogues.
On the one hand, the generalization to twisted K3 surfaces was motivated by the existence of coarse moduli spaces; on the other hand it has become evident that allowing twists has quite unexpected applications (see for example \cite{Huybrechtsderivedtor} for further details).
In \autoref{FGGtwisted} we then consider the twisted picture and extend the above result to this setting. 

Later we consider the ihs manifold of \(\og\) type $\widetilde{X}_Y$ defined by Li, Pertusi, and Zhao (LPZ) which is associated with a smooth cubic fourfold \(Y\) \cite{LPZ}.
We are able to provide a purely lattice-theoretic characterization for $\widetilde{X}_Y$ to be birational to a moduli space of sheaves on a (possibly twisted) K3 surface, providing a new proof of \cite[Theorem 3.2]{GGO}.
\begin{theorem*}[\ref{FGG_implies_GGO} and \ref{FGGtwisted_implies_GGOtwisted}]
Let $\widetilde{X}_Y$ be an LPZ-variety associated to a smooth cubic fourfold $Y$. The following conditions are equivalent.
\begin{enumerate}[(i)]
\item The cubic fourfold $Y$ lies in the Hassett divisor $\mathcal{C}_d$ with $d$ satisfying $\mathrm{(**)}$ (resp. $\mathrm{(**')}$).
\item The LPZ-variety $\widetilde{X}_Y$ is a (resp. twisted) numerical moduli space.
\end{enumerate}
\end{theorem*}
The same question for Fano varieties of lines of cubic fourfolds and for the so called Lehn--Lehn--Sorger--van Straten symplectic eightfolds has been previously answered in \cite{addington, Huybrechtscubic} and \cite{Addingtongiovenzana23,LPZtwisted}. 

Once we know under which conditions an ihs manifold of \(\og\) type is birational to a moduli space of sheaves on a K3 surface $S$, it is then natural to investigate which birational transformations are induced by an automorphism of $S$ (see \autoref{definduced}). To this end we introduce the notion of \textit{numerically induced birational transformations} (see \autoref{defnumericallyinduced}), which is a characterization in terms of their action on the second integral cohomology lattice.

\begin{theorem*}[\ref{induced_aut}]
Let $(X,\eta)$ be a smooth marked pair of $\og$ type and let $G\subset \mathrm{Bir}(X)$ be a finite subgroup. If $G$ is a numerically induced group of birational transformations, then there exists a K3 surface $S$ with an injective group homomorphism $G\hookrightarrow \mathrm{Aut}(S)$, a $G$-invariant Mukai vector $v$ and a $v$-generic polarization $\theta$ on $S$ such that $X$ is birational to $\tmsv$ and $G\subset\mathrm{Bir}(X)$ is an induced group of birational transformations. 
\end{theorem*}

For what concerns symplectic actions on \(\og\) type manifolds we know that the unique symplectic regular automorphism of finite order is the identity \cite[Theorem 1.1]{GGOV}, and symplectic birational involutions are classified \cite{Marquand:sympl.birat.}. Lattice-theoretic constraints for nonsymplectic automorphisms of \(\og\) type are treated in \cite{brandhorst2020prime}. Nonsymplectic automorphisms of \(\og\) obtained from a cubic fourfold via the construction of \cite{LSV} will be treated in the paper in preparation \cite{BG}, where many techniques similar to the ones used to treat automorphisms of the O'Grady six dimensional example (see \cite{GOV20} and \cite{Grossi_nonsympl_OG6}) are exploited. Applying our criterion we show that there exists a unique induced birational symplectic involution on an ihs manifold of \(\og\) type. More precisely we give the following characterization. 

\begin{theorem*}[\ref{example_ind_aut}]
Let \((X, \eta)\) be a marked pair of \(\og\) type. Assume that \(X\) is a numerical moduli space and let \(\varphi \in \Bir(X)\) be a birational symplectic involution. Then \(\varphi\) is induced if and only if
\[H^{2}(X,\mathbb{Z})^{\varphi} \cong \U^{\oplus{3}} \oplus \E_8(-2) \oplus \A_2(-1)  \ \text{and} \ H^{2}(X,\mathbb{Z})_{\varphi} \cong \E_8(-2),\]
where $H^{2}(X,\mathbb{Z})^{\varphi}$ and $H^{2}(X,\mathbb{Z})_{\varphi}$ denote respectively the invariant and the coinvariant lattice under the action of $\varphi$.
\end{theorem*}
\subsection{Outline of the paper}
In \autoref{sec:preliminaries} we recall some preliminary results about lattice theory and briefly describe the lattice structure of the second cohomology group of an ihs manifold. In \autoref{sec:moduliog10} we specialize to manifolds of $\og$ type: after describing their construction and properties, we state and prove \autoref{FGG}. In \autoref{sec:twistedmoduli} we provide a generalization of it to the twisted case.
In \autoref{sec:biratLPZ} we apply the birationality criteria of \autoref{sec:moduliog10} and \autoref{sec:twistedmoduli} to the LPZ variety. Finally in \autoref{sec:numind} we examine birational transformations of $\og$ type to determine which ones are induced. 

\subsection*{Acknowledgements}
The authors wish thank Luca Giovenzana and Giovanni Mongardi for useful discussions. Also, they are indebted with the anonymous referee for taking the time of revising the first first version of the paper so carefully and for the valuable comments on it. 
\subsection*{Fundings}
All authors have been supported by INdAM GNSAGA. \\
 C.F. has been partially supported by PRIN "Moduli spaces and Birational Classifications" and by the University of Modena and Reggio Emilia, project "Discrete Methods in Combinatorial Geometry and Geometric Topology". \\
 F.G. and A.G. have been partially supported by the DFG through the research grant Le 3093/3-2, and partially supported by the European Research Council (ERC) under the European Union’s Horizon 2020 research and innovation programme (ERC-2020-SyG-854361-HyperK). F.G.’s research have been partially funded by Deutsche Forschungsgemeinschaft (DFG, German Research Foundation), Projektnummer 509501007.

\section{Preliminaries in lattice theory}\label{sec:preliminaries}
In this section we recall some basic facts of lattice theory for irreducible holomorphic symplectic manifolds, focusing on the special case of manifolds of $\og$ type. 

\subsection{Abstract lattices}
\begin{definition} A lattice $\bL$ is a free finite rank $\mathbb{Z}$-module with a non-degenerate symmetric bilinear form 
\begin{align*}
\cdot \colon \bL\times \bL &\rightarrow \mathbb{Z}\\
(e,f) &\mapsto e\cdot f.
\end{align*}
\end{definition}
With an abuse of notation we denote $e\cdot e$ simply by $e^2$. Moreover, we say that the lattice $\bL$ is \textit{even} if $e^2$ is even for any element \(e \in \bL\) and we denote the determinant, the rank and the signature of the lattice $\bL$ by $\det(\bL)$, $\rank(\bL)$ and $\Sign(\bL)$ respectively. 
\begin{definition}
Given $e\in \bL$, we define the \textit{divisibility} of $e$ in $\bL$ as
\[(e,\bL)\coloneqq \mathrm{gcd}\{e\cdot f\mid f\in \bL\}.\]
\end{definition}
Given a lattice $\bL$, we define the \textit{dual lattice} of $\bL$ to be the lattice $\bL^{\vee}\coloneqq \Hom_{\mathbb{Z}}(\bL,\mathbb Z)$. The dual lattice can be characterized uniquely as 
$$ \bL^{\vee}=\left\lbrace x\in \bL\otimes  \mathbb{Q}\mid x\cdot y\in \mathbb{Z} \text{ for all }y\in \bL \right\rbrace.$$

One can show that the quotient $A_{\bL}\coloneqq\bL^{\vee}/\bL$ is a finite group, called \textit{discriminant group} of $\bL$, of order $\vert\det(\bL)\vert$. We say that a lattice $\bL$ is \textit{unimodular} if $A_{\bL}$ is trivial.

\begin{notation}
We denote the unique unimodular even indefinite rank 2 lattice by $\U$. Moreover the symbols $\A_n,\D_n,\E_n$ denote the positive definite ADE lattices of rank $n$. The notation $[m]$ for some $m\in \mathbb{Z}$ indicates a rank 1 lattice generated by a vector of square $m$. \\
For any lattice $\bL$ we denote by $\bL(n)$ the lattice with the same structure as $\bL$ as a $\mathbb{Z}$-module and quadratic form multiplied by $n$.
\end{notation}

Whenever we have a lattice $\bL$, we can consider a subgroup $G\subset O(\bL)$ of isometries on $\bL$. We denote by $\bL^G$ the \textit{invariant} sublattice of $\bL$ with respect to the action of $G$, and by $\bL_{G}\coloneqq (\bL^{G})^{\perp}$ its orthogonal complement in $\bL$, which is called \textit{coinvariant} sublattice.

\subsubsection{Hodge structure on $\bL$} Given a lattice $\bL$ with a weight 2 Hodge structure, i.e. a decomposition 
\[\bL_{\mathbb{C}}\coloneqq \bL\otimes_{\mathbb{Z}}\mathbb{C}=\bL_{\mathbb{C}}^{2,0}\oplus \bL_{\mathbb{C}}^{1,1}\oplus \bL_{\mathbb{C}}^{0,2} \]
with $\overline{\bL_{\mathbb{C}}^{2,0}}\cong \bL_{\mathbb{C}}^{0,2}$ and $\overline{\bL_{\mathbb{C}}^{1,1}}\cong \bL_{\mathbb{C}}^{1,1}$, 
we set 
\[\bL^{1,1}\coloneqq \bL_{\mathbb{C}}^{1,1}\cap \bL,\]
and refer to it as the \((1,1)\) - part of the lattice $\bL$.
\begin{definition}
An embedding of lattices $\mathbf{M}\hookrightarrow \bL$ is called \textit{primitive} if the group $\bL/\mathbf{M}$ is torsion free. In this setting we denote by $\mathbf{M}^{\perp}$ the orthogonal complement of $\mathbf{M}$ in $\bL$ with respect to the bilinear form on $\bL$.
\end{definition}

\begin{remark}
If $\bL$ and $\mathbf{M}$ are two lattices with pairings respectively denoted by $\cdot_{\bL}$ and $\cdot_{\mathbf{M}}$ then the $\mathbb{Z}$-module $\bL\oplus\mathbf{M}$ inherits a lattice structure with the pairing $\cdot_{\bL\oplus \mathbf{M}}$ given by 
\[e\cdot_{\bL\oplus \mathbf{M}}f=\begin{cases}
e\cdot_{\bL}f \ \mbox{ if $e,f\in \bL$}\\
e\cdot_{\mathbf{M}}f \  \mbox{ if $e,f\in \mathbf{M}$}\\
0 \ \mbox{ otherwise}
\end{cases}\]
Following \cite{PR14}, we denote this lattice by $\bL\oplus_{\perp}\mathbf{M}$.
\end{remark}

\begin{definition}
Let $\mathbf{M}$ be a lattice endowed with a weight 2 Hodge structure. An embedding of lattices $\mathbf{M}\hookrightarrow \bL$ is called $\textit{Hodge embedding}$ if it is primitive and $\bL$ is endowed with a weight 2 Hodge structure defined as follows:
\[\bL_{\mathbb{C}}^{2,0}=\mathbf{M}_{\mathbb{C}}^{2,0}, \quad \bL_{\mathbb{C}}^{1,1}=(\mathbf{M}^{1,1}\oplus \mathbf{M}^{\perp})\otimes \mathbb{C},\quad \bL_{\mathbb{C}}^{0,2}=\mathbf{M}_{\mathbb{C}}^{0,2}.\]
Moreover, if $\bL$ has already a Hodge structure on it, we say that the Hodge embedding $\mathbf{M}\hookrightarrow \bL$ is \textit{compatible} if the Hodge structure induced by $\mathbf{M}$ is the one of $\bL$.
\end{definition}

\subsection{Lattice structure of irreducible holomorphic symplectic manidolds}
Suppose $X$ is an irreducible holomorphic symplectic manifold of complex dimension $2n$. The second integral cohomology group has a well defined lattice structure: in fact, it is a torsion free $\mathbb{Z}$-module of finite rank endowed with a symmetric bilinear form $(-,-)_X$, called the \textit{Beauville-Bogomolov-Fujiki} (BBF) form, which satisfies the following equality for any $\alpha\in H^2(X,\mathbb{Z})$:
\[\int_X\alpha^{2n}= c_X(\alpha,\alpha)^n_X,\]
for some unique positive rational constant, called the \textit{Fujiki constant}, see for example \cite{beauville1983varietes, fujiki1983primitively}.
\begin{remark}\label{definv}
Note that both the BBF form and the Fujiki constant are invariant up to deformation. 
\end{remark}

\section{Moduli spaces of sheaves of O'Grady 10 type}\label{sec:moduliog10}

In his seminal paper \cite{Ogrady99}, O'Grady discovers a 10-dimensional ihs manifold not deformation equivalent to a Hilbert scheme of K3 or a generelized Kummer variety. The example arises as a symplectic resolution of a moduli space of sheaves on a K3 surface with fixed numerical constraints.
In \autoref{sub:Muk lattice} and in \autoref{sub:moduli} we recall the original construction and its generalizations, focusing also on the lattice-theoretic point of view. In \autoref{subsec:birat_crit} we are finally ready to state and prove criteria to determine if a manifold of \(\og\) type is birational to a moduli space of sheaves on a K3 surface. 
\subsection{The Mukai lattice}\label{sub:Muk lattice}

Let $S$ be a K3 surface. We denote by $\widetilde{H}(S,\mathbb{Z})$ the even integral cohomology of $S$, that is 
\[\widetilde{H}(S,\mathbb{Z})=H^{2*}(S,\mathbb{Z})=H^{0}(S,\mathbb{Z})\oplus H^{2}(S,\mathbb{Z})\oplus H^{4}(S,\mathbb{Z}).\]
We can put a symmetric bilinear pairing on it, called the \textit{Mukai pairing}, in the following way:
\begin{align*}
\widetilde{H}(S,\mathbb{Z})\times \widetilde{H}(S,\mathbb{Z})&\rightarrow \mathbb{Z}\\
(r_1,l_1,s_1),(r_2,l_2,s_2) &\mapsto -r_1s_2+l_1l_2-r_2s_1
\end{align*}
with $r_i\in H^{0}(S,\mathbb{Z})$, $l_i\in H^{2}(S,\mathbb{Z})$ and $s_i\in H^{4}(S,\mathbb{Z})$.

 The lattice $\widetilde{H}(S,\mathbb{Z})$ together with the Mukai pairing is referred to as the \textit{Mukai lattice} and any element $v=(v_0,v_1,v_2)\in \widetilde{H}(S,\mathbb{Z})$ with $v_i\in H^{2i}(S)$ is called \textit{Mukai vector}.

The Mukai lattice is isometric to the unique unimodular lattice of rank 24 and signature \((4,20)\). We denote this isometry class by $\bLambda_{24}$ and we have
\[\bLambda_{24}=\U^{\oplus 4}\oplus \E_8(-1)^{\oplus 2}.\] The
natural Hodge structure on $H^2(S,\mathbb{Z})$ can be extended to a weight 2 Hodge structure on $\widetilde{H}(S,\mathbb{Z})$ by setting $\widetilde{H}_{\mathbb{C}}^{2,0}(S)\coloneqq H^{2,0}_{\mathbb{C}}(S)$ (resp. $\widetilde{H}_{\mathbb{C}}^{0,2}(S)\coloneqq H_{\mathbb{C}}^{0,2}(S)$) and
\[\widetilde{H}^{1,1}_{\mathbb{C}}(S)\coloneqq H^0(S,\mathbb{C})\oplus H_{\mathbb{C}}^{1,1}(S)\oplus H^4(S,\mathbb{C}). \]
Note that, for any Mukai vector $v\in \widetilde{H}(S,\mathbb{Z})$, the orthogonal sublattice with respect to the Mukai pairing  
\[v^{\perp}=\{w\in \widetilde{H}(S,\mathbb{Z})\mid (v,w)=0\}\subseteq \widetilde{H}(S,\mathbb{Z}) \]
inherits a weight 2 Hodge structure from $\widetilde{H}(S,\mathbb{Z})$ in an obvious fashion.

Moreover, given a coherent sheaf $\mathcal{F}$ on $S$, we can naturally associate a Mukai vector to it by setting
\[v(\mathcal{F})\coloneqq(\mathrm{rank} (\mathcal{F}), c_1(\mathcal{F}),c_1(\mathcal{F})^2/2 -c_2(\mathcal{F}) + \mathrm{rank}(\mathcal{F}))\in \widetilde{H}(S,\mathbb{Z}).\]
By construction, $v(\mathcal{F})=(r,l,s)$ is of (1,1)-type and it satisfies one of the following relations:
\begin{enumerate}[(i)]
\item $r>0$;
\item $r=0$ and $l\neq 0$ with $l$ effective;
\item $r=l=0$ and $s>0$.
\end{enumerate}

\begin{definition} A nonzero vector $v\in \widetilde{H}(S,\mathbb{Z})$ satisfying $v^2\geq 2$ and one of the conditions above is called \textit{positive} Mukai vector.
\end{definition}

\subsection{Moduli spaces of sheaves of $\og$ type}\label{sub:moduli}
Given an algebraic Mukai vector $v$ and a $v$-generic polarization $\theta$, we can consider the moduli space $\msv$ of $\theta$-semistable sheaves on $S$ with Mukai vector $v$.

More generally we consider Mukai vectors $v=mw$ where $m\in \mathbb{N}$ and $w\in \widetilde{H}(S,\mathbb{Z})$ is some primitive Mukai vector. Let $\msvs \subseteq \msv$ be the locus of $\theta$-stable sheaves: it is well known that if $\msvs\neq \emptyset$ then it is smooth of dimension $v^2+2$ and carries a symplectic form \cite{mukai1984symplectic,Yoshioka03,Yoshioka99}.

\begin{theorem}[Mukai, Yoshioka] Let $S$ be a K3 surface, $v$ be a primitive and positive Mukai vector
and $\theta$ be a $v$-generic polarization. Then $\msvs=\msv$ and it is a irreducible holomorphic symplectic manifold of dimension $v^2+2=2n$, which is deformation equivalent to the $n$-th Hilbert scheme of a K3 surface. Moreover there is a Hodge isometry between $v^{\perp}\cong H^2(\msv,\mathbb{Z})$.
\end{theorem}

When $v$ is not primitive, $\msv$ can be singular with a holomorphic symplectic form on its smooth locus. It is then natural to investigate under which conditions $\msv$ admits a symplectic resolution of singularities, namely a resolution carrying a holomorphic symplectic form which extends the holomorphic symplectic form on the smooth locus of $\msv$.
O'Grady starts from this consideration and proves that the moduli space $\mathcal{M}_{10}\coloneqq\msv$ with $v=(2,0,-2)$ and $\theta$ a $v$-generic polarization admits a symplectic resolution $\pi:\widetilde{\mathcal{M}}_{10}\rightarrow \mathcal{M}_{10}$ such that  $\widetilde{\mathcal{M}}_{10}$ is
an irreducible holomorphic symplectic manifold of dimension 10 and second Betti number 24. This latter condition implies in particular that $\widetilde{\mathcal{M}}_{10}$ cannot be deformation equivalent to a Hilbert scheme, so we have a new deformation class of ihs manifolds. 

\begin{definition}[\cite{LLS06}]
We say that an ihs manifold $X$ is of $\og$ type if it is deformation equivalent to the O'Grady 10 dimensional example $\widetilde{\mathcal{M}}_{10}$.
\end{definition} 

In \cite{Rapog10}, Rapagnetta describes the lattice structure of $\widetilde{\mathcal{M}}_{10}$.
Let $\Sigma$ be the singular locus of $\mathcal{M}_{10}$, $B\subset\mathcal{M}_{10}$ be the locus parametrizing non locally free sheaves and let $\widetilde{\Sigma}$ and $\widetilde{B}$ be respectively the exceptional divisor of $\pi$ and the strict transform of $B$.
Finally, let 
\[\widetilde{\mu}:H^2(S,\mathbb{Z})\rightarrow H^2(\widetilde{\mathcal{M}}_{10},\mathbb{Z})\] be the Donaldson morphism, see for example \cite{Li93, Mo93,FM94}. 
\begin{theorem}[{\cite[Theorems 2.0.8 and 3.0.11]{Rapog10}}]
In the above notation we have:
\begin{enumerate}[(i)]
\item the morphism $\widetilde{\mu}$ is injective and 
\[H^2(\widetilde{\mathcal{M}}_{10},\mathbb{Z})=\widetilde{\mu}(H^2(S,\mathbb{Z}))\oplus c_1(\widetilde{\Sigma})\oplus c_1(\widetilde{B})\]
\item The map $\widetilde{\mu}$ is an isometry with respect to the intersection pairing on $S$ and the BBF form on $H^2(\widetilde{\mathcal{M}}_{10})$.
\item Setting $\Delta=c_1(\widetilde{\Sigma})\oplus c_1(\widetilde{B})$, we have that the decomposition 
 \[ H^2(\widetilde{\mathcal{M}}_{10},\mathbb{Z})=\widetilde{\mu}(H^2(S,\mathbb{Z}))\oplus_{\perp} \Delta\]
 is orthogonal with respect to the BBF form and that the restriction of the BBF form on $\Delta$ is 
 \begin{center}
 \begin{tabular}{|c|c|c|}
\hline
 & $c_1(\tilde{\Sigma})$ & $c_1(\tilde{B})$\\
\hline
$c_1(\tilde{\Sigma})$ & -6 & 3 \\
\hline
$c_1(\tilde{B})$ & 3 & -2 \\
\hline
\end{tabular}
\end{center}
\end{enumerate}
\end{theorem}

The construction has been later generalized by Perego and Rapagnetta \cite{PR13} allowing different types of Mukai vectors and polarizations. More precisely, Perego and Rapagnetta introduce the concept of OLS-triple $(S,v,\theta)$ where $S$ is a K3 surface, $v$ is an algebraic Mukai vector and $\theta$ a $v$-generic polarization.
\begin{definition}[\cite{PR13}, Definition 1.5]
Let $S$ be a projective K3 surface, $\theta$ a polarization on $S$ and $v$ an algebraic Mukai vector. We say that $(S,v,\theta)$ is an OLS-triple if 
\begin{enumerate}[(i)]
\item $\theta$ is primitive and $v$-generic;
\item $v=2w$ where $w$ is a primitive Mukai vector with $w^2=2$;
\item if we write $w=(r,l,s)$ then $r\geq 0$, $l\in \mathrm{NS}(S)$ and if $r=0$ then $l$ is the first Chern class of an effective divisor. 
\end{enumerate}
\end{definition}

One of the main results of \cite{PR13} asserts that whenever $(S,v,\theta)$ is an OLS-triple, the moduli space $\msv$ admits a symplectic resolution $\tmsv$ which is of \(\og\) type.
\begin{theorem}[\cite{PR13},Theorem 1.6]
Let $(S,v,\theta)$ be an OLS-triple. The resolution $\tmsv$ of the moduli space $\msv$ is an irreducible holomorphic symplectic manifold which is deformation equivalent to $\widetilde{\mathcal{M}}_{10}$.
\end{theorem}

As a consequence, for any OLS triple $(S,v,\theta)$ the second integral cohomology group of $\tmsv$ is endowed with a lattice structure with the BBF form and with a weight 2 Hodge structure.
Since $\msv$ has rational singularities, the pullback $\pi^*:H^2(\msv,\mathbb{Z})\rightarrow H^2(\tmsv,\mathbb{Z})$ via is injective and by strict compatibility of weight filtrations $H^2(\msv,\mathbb{Z})$ admits a pure weight 2 Hodge structure. Also, it inherits a lattice structure by restricting the BBF form on $H^2(\tmsv,\mathbb{Z})$.

\begin{theorem}[\cite{PR13},Theorem 1.7]\label{pr1.7}
Let $(S,v,\theta)$ be an OLS-triple. 
There is an isometry of Hodge structures 
\[\lambda_v: v^{\perp}\xrightarrow{\simeq} H^2(\msv,\mathbb{Z}).\]
\end{theorem}

Clearly, if we denote by \[\widetilde{\lambda_v}=\pi^*\circ \lambda_v: v^{\perp}\hookrightarrow H^2(\tmsv,\mathbb{Z})\]
the composition of the above isometry with the pullback, we have that $\widetilde{\lambda_v}$ is an injective morphism of Hodge structures and an isometry onto its image. It is possible to extend this morphism to a Hodge isometry of lattices as follows, see also \cite[\S 3]{PR14} for proofs and further details.
Consider the dual lattice 
\[(v^{\perp})^\vee=\{ \alpha \in v^{\perp}\otimes_{\mathbb{Z}}\mathbb{Q}\mid (\alpha,\beta)\in \mathbb{Z} \ \forall \beta \in v^{\perp}\}.\]
The non degeneracy of the Mukai pairing implies that the morphism
\[v^{\perp}\rightarrow (v^{\perp})^{\vee},\ \alpha\mapsto (\alpha,-)\]
is injective and thus that $v^{\perp}$ can be viewed as a sublattice of $(v^{\perp})^\vee$ of index $|\mathrm{det}(v^{\perp})|=2$.
We denote by $( - , - )_{\mathbb{Q}}$ the $\mathbb{Q}$-bilinear extension of the Mukai pairing to $(v^{\perp})^{\vee}$.
Next, one considers the $\mathbb{Z}$-module
\[(v^{\perp})^\vee\oplus \mathbb{Z}\sigma/2\]
and endows it with a symmetric $\mathbb{Q}$-bilinear form
$b_v$ such that 
\begin{itemize}
\item $b_{v\mid (v^{\perp})^\vee}=( - , - )_{\mathbb{Q}}$,
\item $\sigma$ is orthogonal to $(v^{\perp})\vee$,
\item $b_v(\sigma,\sigma)=-6$.
\end{itemize}
If we consider the submodule $\Gamma_v$ of $(v^{\perp})^\vee\oplus \mathbb{Z}\sigma/2$
\[\Gamma_v:=\{(\alpha,k\sigma/2)\mid k\in 2\mathbb{Z}\text{ if and only if } \alpha \in v^{\perp}\},\]
we can endow it with a weight 2 Hodge structure by setting
\[(\Gamma_v)^{2,0}=(v^{\perp})^{2,0},\quad (\Gamma_v)^{2,0}=(v^{\perp})^{1,1}\oplus \mathbb{C}\sigma,\quad (\Gamma_v)^{0,2}=(v^{\perp})^{0,2}.\]
\begin{theorem}[\cite{PR14}, Theorem 3.4]\label{h2tmsv}
The restriction of $b_v$ to $\Gamma_v$ takes integral values (thus endowing $\Gamma_v$ with a lattice structure) and there is a Hodge isometry
$$ f_v: \Gamma_v\xrightarrow{\simeq} H^2(\tmsv,\mathbb{Z}), \quad f_v(\alpha,k\sigma/2)=\widetilde{\lambda_v}(\alpha)+kE,$$
where $E$ is the class of the exceptional divisor of $\pi$.
\end{theorem}

From the previous discussion, one can easily deduce the following corollary. 
\begin{corollary}\label{cor:finiteindexemb}
In the above notation, there exists a finite index embedding
\[v^{\perp}\oplus_{\perp} \mathbb{Z}E\hookrightarrow H^2(\tmsv,\mathbb{Z}).\]
\end{corollary}

The proof of \autoref{h2tmsv} is first carried out for original O' Grady example where $v=(2,0,-2)$ (see \cite[Remark 3.2]{PR14}) and then extended to all OLS triples via a deformation argument. 
Namely Perego and Rapagnetta show that the Hodge morphism $f_v$ takes values in $H^2(\tmsv,\mathbb{Z})$, that is a morphism of $\mathbb{Z}$-modules and an isometry of lattices. 
A key point in the proof is showing that under the chosen deformation, all these properties are preserved. 
In particular the class $E$ of the exceptional divisor in $H^2(\tmsv,\mathbb{Z})$ corresponds to $c_1(\widetilde{\Sigma})\in H^2(\widetilde{\mathcal{M}_{10}})$, thus it has square -6 and divisibility 3 (see \cite[Remark 3.5]{PR14}).

\subsubsection{Lattice of O'Grady 10 type manifolds}

In fact, one can describe the lattice structure for any variety of $\og$ type.
\begin{definition}
    A marked pair of $\og$ type is pair $(X,\eta)$ where $X$ is an irreducible holomorphic symplectic manifold of $\og$ type and \(\eta\colon H^{2}(X,\mathbb{Z}) \to \bL\) is a fixed isometry of lattices.
\end{definition}
 In what follows, we denote by $\bL$ the isometry class of $H^{2}(X,\mathbb{Z})$, which is well defined in view of \autoref{definv}.
In \cite{Rapog10} Rapagnetta shows that $\bL$ is isometric to
\[\U^{\oplus 3}\oplus \E_8(-1)^{\oplus 2}\oplus \A_2(-1),\]
which is an even lattice of rank $24$ and signature \((3,21)\).

Moreover, being the second cohomology group of a compact K\"ahler manifold, $\bL$ is endowed with a natural weight \(2\) Hodge structure, whose Hodge numbers are deformation invariant.

In particular, if $(X,\eta)$ is a marked pair of $\og$ type, the integral lattice $\bL^{1,1}$ is called the \textit{Néron--Severi} lattice of $X$ and it is denoted by $\NS(X)$. Its orthogonal complement is called \textit{transcendental lattice} and it is denoted by $\T(X)$.
\begin{remark}
Note that in the notation above, if \(\sigma \in \bL\) is a primitive class such that \(\sigma^2=-6\) and \((\sigma, \bL)=3\), then the lattice \(\sigma^{\perp} \cong \U^{\oplus3} \oplus \E_8(-1)^{\oplus3} \oplus [-2]\) admits a primitive embedding in the even unimodular lattice \(\bLambda_{24}\) sending the unimodular part in itself and the generator of square \(-2\) in the difference of two generators of one copy of \(\U\). Moreover this primitive embedding is unique, see \cite[Proposition 1.15.1]{Nik79}. 
\end{remark}

\subsection{Birationality criteria for OG10 type manifolds}\label{subsec:birat_crit}

We provide an entirely lattice-theoretic criterion to determine when a manifold $X$ of $\og$ type is birational to a moduli space $\tmsv$ of sheaves  on a K3 surface $S$, where $(S,v, \theta)$ is an OLS-triple as described in \autoref{sec:preliminaries}.

\begin{definition}\label{nms}
Let $(X,\eta)$ be a projective marked pair of $\og$ type, where $\eta \colon H^2(X,\mathbb{Z})\rightarrow \bL$ is a fixed marking. We say that $X$ is a \textit{numerical moduli space} of $\og$ type if there exists a primitive class $\sigma\in \bL^{1,1}$ such that 
\begin{enumerate}[(a)]
\item $\sigma^2=-6$ and $(\sigma,\bL)=3$,
\item The Hodge embedding $\sigma^{\perp_{\bL}}\hookrightarrow \bLambda_{24}$ embeds a copy of $\U$ in $\bLambda_{24}^{1,1}$.

\end{enumerate}
\end{definition}
We are now in a position to state our birationality criterion.
\begin{theorem}\label{FGG}
Let $(X,\eta)$ be a marked pair of $\og$ type. The following conditions are equivalent:
\begin{enumerate}[(i)]
    \item There exists a K3 surface $S$, a Mukai vector $v=2w$ where $w$ is a primitive Mukai vector of square $2$ and a $v-$generic polarization $\theta$ on $S$ such that $X$ is birational to $\tmsv$;
    \item The manifold \(X\) is a numerical moduli space.
\end{enumerate}
\end{theorem}
\begin{proof}
Suppose $\Phi:X\dashrightarrow \tmsv$ is a birational morphism. Then the induced isometry $\Phi^*: H^2(\tmsv,\mathbb{Z})\rightarrow H^2(X,\mathbb{Z})$ is an isometry of Hodge structures.
Recall that the manifold $\tmsv$ is obtained as a symplectic resolution of singularities of the moduli space $\msv$. 

Note that the class of the exceptional divisor $E\in H^2(\tmsv,\mathbb{Z})$ is primitive of (1,1) type, of square \(-6\) and divisibility \(3\). Moreover, by \autoref{cor:finiteindexemb}, we have that $$E^{\perp_{\bL}}\cong H^2(\msv,\mathbb{Z})$$ and there exists a Hodge embedding $H^2(\msv,\mathbb{Z})\hookrightarrow \widetilde{H}(S,\mathbb{Z})\cong \bLambda_{24}$.
Let $\sigma:=\Phi^*E$. We now prove that $X$ is a numerical moduli space: since $\Phi^*$ is an isometry of Hodge structures, condition (a) in \autoref{nms} is satisfied. 
For condition (b), observe that the composition
$$\sigma^{\perp}\xrightarrow{(\Phi^{-1})^*} {E^{\perp}}\cong H^2(\msv,\mathbb{Z})\hookrightarrow \widetilde{H}(S,\mathbb{Z})\cong \bLambda_{24}\cong \U^{\oplus 4}\oplus \E_8(-1)^{\oplus 2}$$
is a compatible Hodge embedding. As a result the weight 2 Hodge structure on \(\widetilde{H}(S,\mathbb{Z})\) induced by the embedding $\sigma^{\perp}\hookrightarrow \bLambda_{24}$ is the one defined on the Mukai lattice in \autoref{sub:Muk lattice}. The vectors \((1,0,0)\) and \((0,0,1)\) generate a copy of \(\U\) in $\widetilde{H}(S,\mathbb{Z})^{1,1}\cong\bLambda_{24}^{1,1}$, hence \(X\) is a numerical moduli space.

 We now prove the implication in the other direction. Supppose that $X$ is a numerical moduli space, i.e. that there exists a class $\sigma \in \bL^{1,1}$ such that $\sigma^2=-6$ and $(\sigma,\bL)=3$, and a Hodge isometry $\sigma^{\perp_{\bL}}\hookrightarrow \bLambda_{24}$ such that $\bLambda_{24}^{1,1}$ contains a copy of $\U$ as a direct summand.\\
Observe that the lattice $\sigma^{\perp_{\bL}}$ is isometric to the lattice $\U^{\oplus 3}\oplus \E_8(-1)^{\oplus 2}\oplus [-2]$, so that setting 
$$w:=(\sigma^{\perp_{\bL}})^{\perp_{\bLambda_{24}}},$$
we have that $w^2=2$. Note that, by definition, $[w]$ is a sublattice of $\bLambda_{24}$ and it is of $(1,1)$ type.
By construction, we have
$$w^{\perp_{\bLambda_{24}}}\cong \U^{\oplus 3}\oplus \E_8(-1)^{\oplus 2}\oplus [-2]\cong \sigma^{\perp_{\bL}}$$
so $\mathrm{sgn}(w^{\perp_{\bLambda_{24}}})=\mathrm{sgn}(\sigma^{\perp_{\bL}})=(3,20)$. Moreover the projectivity of $X$ implies the positive part of the signature of $\mathrm{NS}(X)\cong \bL^{1,1}$ is equal to 1.
Since the class $\sigma$ has negative square, the positive part of $\mathrm{sgn}((\sigma^{\perp_{\bL}})^{1,1})$ is equal to the positive part of the signature of $\bL^{1,1}$, which is 1. \\
The lattices $\bLambda_{24}^{1,1}$ and $(\sigma^{\perp_{\bL}})^{1,1}\oplus [w]$ have the same signature: in particular the positive part of $\mathrm{sgn}(\bLambda_{24}^{1,1})$ is equal to 2.


By hypothesis
$$\bLambda_{24}^{1,1}\cong \U\oplus \boldsymbol{\mathrm{T}}$$
for some even lattice $\boldsymbol{\mathrm{T}}$ of signature \((1,-)\).
By Torelli theorem for K3 surfaces, there exists a K3 surface $S$ such that $\mathrm{NS}(S)\cong \boldsymbol{\mathrm{T}}$. Setting $v\colon=2w$, we can always choose a positive generator $\theta$ which is a  $v$-generic polarization on $S$. 
Now since $w^2=2$, by \cite[Lemma 1.28]{MW15} either $w$ or $-w$ is a positive Mukai vector, hence we may assume that $\left(S,v,\theta\right)$ is an OLS-triple.
By \autoref{pr1.7} there exists a Hodge isometry
$$ H^2(\msv,\mathbb{Z})\cong v^{\perp}\subset \bLambda_{24},$$
where the orthogonal complement of $v$ is computed with respect to the Mukai pairing.
The singular moduli space $\msv$ admits a symplectic resolution $\tmsv$ with exceptional divisor $E\in H^2(\tmsv,\mathbb{Z})$ such that $E$ has divisibility 3 and $E^2=-6$.

We prove that $X$ is birational to $\tmsv$ by showing that there exists a Hodge isometry between $H^2(X,\mathbb{Z})$ and $H^2(\tmsv,\mathbb{Z})$. 
In fact, when such an isometry exists, birationality is implied by the maximality of $\Mon^2(\og)$ for \(\og\) type manifolds, see \cite{onorati2020monodromy} and \cite[Theorem 5.2(2)]{MROG6mon}.\\
Observe that we have a Hodge isometry
$$ \alpha: H^2(\msv,\mathbb{Z})\oplus \mathbb{Z} E\cong \sigma^{\perp_{\bL}}\oplus \mathbb{Z}\sigma$$ 
and two finite index embeddings
\begin{align}\label{emb}
      H^2(\msv,\mathbb{Z})\oplus \mathbb{Z} E &\xhookrightarrow{i} H^2(\tmsv,\mathbb{Z})\\
     \sigma^{\perp_{\bL}}\oplus \mathbb{Z}\sigma &\xhookrightarrow{j} H^2(X,\mathbb{Z}).
\end{align}
We want to lift the Hodge isometry $\alpha$ to a Hodge isometry $\beta:H^2(\tmsv,\mathbb{Z})\cong H^2(X,\mathbb{Z})$ so that the following diagram commutes:
\[ \begin{tikzcd} H^2(\msv,\mathbb{Z})\oplus \mathbb{Z} E \arrow[hook]{r}{i} \arrow{d}[']{\alpha} & H^2(\tmsv,\mathbb{Z})\arrow{d}{\beta}\\ \sigma^{\perp_{\bL}}\oplus \mathbb{Z}\sigma \arrow[hook]{r}{j} &H^2(X,\mathbb{Z}).\end{tikzcd}\]
Observe that, as abstract lattices, $H^2(\tmsv,\mathbb{Z})$ and $H^2(X,\mathbb{Z})$ are both isomorphic to 
$M:=\U^{\oplus 3}\oplus \E_8(-1)^{\oplus 2}\oplus \A_2(-1)$.
The existence of $\beta$ follows from a result by Nikulin \cite[Corollary 1.5.2]{Nik79} and a direct computation, by setting (in the loc. cit. notation)
\begin{itemize}
\item $S_1= H^2(\msv,\mathbb{Z})$ and $K_1=S_1^{\perp}=\mathbb{Z}E$
\item $S_2=\sigma^{\perp_{\bL}}$ and $K_2=S_2^{\perp}=\mathbb{Z}\sigma$.
\item $\varphi=\alpha_{\mid H^2(\msv,\mathbb{Z})}$ and $\psi(E):=\sigma$.
\end{itemize}
Then $\beta$ is a Hodge isometry by construction and the proof is concluded.
\end{proof}

\section{Moduli spaces of twisted sheaves of O'Grady 10 type}\label{sec:twistedmoduli}
In this section we recall basic facts about twisted $K3$ surfaces and extend the results of \autoref{sec:moduliog10}. For more details on twisted $K3$ surfaces we refer to \cite{Huybrechts_Stellari}.
\begin{definition}
A twisted K3 surface is a pair \((S, \alpha)\) given by a K3 surface \(S\) and a Brauer class \(\alpha \in \Br(S)\), where \(\Br(S)\) denotes the torsion part of \(H^2(S,\mathcal{O}_S^{*})\). 
\end{definition}
 Choosing a lift \(B \in H^2(S, \mathbb{Q})\) of \(\alpha\) under the natural morphism \(H^2(S, \mathbb{Q}) \to \Br(S)\) induced by the exponential sequence allows one to introduce a natural Hodge structure \(\widetilde{H}(S,\alpha,\mathbb{Z})\) of weight two associated with \((S,\alpha)\). As a lattice, this is just $\widetilde{H}(S, \mathbb{Z}) \cong H^{0}(S,\mathbb{Z})\oplus H^{2}(S,\mathbb{Z}) \oplus H^4(S,\mathbb{Z})$, but the $(2, 0)$ - part is now given by \[\widetilde{H}^{2,0} (S,\alpha, \mathbb{Z}) \coloneqq \mathbb{C}(\omega + \omega \wedge B)\] where \(0 \neq \omega\in H^{2,0}(S)\). This defines a Hodge structure by setting 
 \[ \widetilde{H}^{0,2}(S,\alpha)\coloneqq \overline{\widetilde{H}^{2,0}}, \quad \widetilde{H}^{1,1}(S, \alpha)\coloneqq (\widetilde{H}^{2,0}(S, \alpha)\oplus \widetilde{H}^{0,2}(S,\alpha))^{\perp}\] where orthogonality is with respect to the Mukai pairing. Although the definition depends on the choice of \(B\), the Hodge structures induced by two different lifts \(B\) and \(B'\) of the same Brauer class \(\alpha\) are Hodge isometric \cite{Huybrechts_Stellari}.

When $S$ is a K3 surface and $\alpha\in \mathrm{Br}(S)$, we get a well defined notion of coherent sheaf on the twisted K3 surface $(S,\alpha)$ by replacing the identity with $\alpha$ in the cocycle condition, see \cite[\S 16.5]{huybrechts2016lectures}.
By the twisted version of derived global Torelli theorem, the bounded derived categories of coherent sheaves on two twisted K3 surfaces \((S,\alpha)\) and \((S',\alpha')\) are equivalent if and only if there exists a Hodge isometry \(\widetilde{H}(S, \alpha, \mathbb{Z}) \cong \widetilde{H}(S', \alpha', \mathbb{Z})\) preserving the natural orientation of the four positive directions, see \cite{SH04, Reinecke}.

We want to generalise our criterion to the case of moduli space of sheaves $\tmsvtwisted$ on a twisted K3 surface $(S,\alpha)$. To this end, we introduce the notion of twisted numerical moduli space. 
\begin{definition}\label{nmstwisted}
Let $(X,\eta)$ be a projective marked pair of $\og$ type, where $\eta \colon H^2(X,\mathbb{Z})\rightarrow \bL$ is a fixed marking. We say that $X$ is a twisted numerical moduli space of $\og$ type if there exists a primitive class $\sigma\in \bL^{1,1}$ such that 
\begin{enumerate}[(a)]
\item $\sigma^2=-6$ and $(\sigma,\bL)=3$,
\item the Hodge embedding $\sigma^{\perp_{\bL}}\hookrightarrow \bLambda_{24}$ embeds a copy of $\U(n)$ in $\bLambda_{24}^{1,1}$ as a direct summand for some positive $n\in \mathbb{N}$.
\end{enumerate}
\end{definition}

We can now prove an analogue version of \autoref{FGG} for the twisted case.

\begin{theorem}\label{FGGtwisted}
Let $(X,\eta)$ be a marked pair of $\og$ type. The following conditions are equivalent:
\begin{enumerate}[(i)]
    \item There exists a twisted K3 surface $(S,\alpha)$, a non-primitive Mukai vector $v=2w$  and a $v-$generic polarization $\theta$ on $S$ such that $X$ is birational to $\tmsvtwisted$;
    \item The manifold X is a twisted numerical moduli space.
\end{enumerate}
\end{theorem}
\begin{proof}
As before, suppose $X$ is birational to a symplectic resolution $\tmsvtwisted$ of a moduli space $\msvtwisted$ for some twisted K3 surface $(S,\alpha)$, a non-primitive Mukai vector $v=2w$ and a $v-$generic polarization $\theta$ on $S$. Then we have a Hodge isometry
\begin{equation}\label{eq:hodgeisomtwisted}
H^2(X,\mathbb{Z})\cong H^2(\tmsvtwisted,\mathbb{Z}),
\end{equation}
such that $H^2(\msvtwisted,\mathbb{Z})$ embeds in $H^2(\tmsvtwisted,\mathbb{Z})$ as the orthogonal complement of the exceptional divisor $E$, which is of divisibility \(3\) and square \(-6\). 
Moreover, there is a compatible Hodge embedding by \cite[Proof of Theorem 2.7]{Meachan_Zhang} $$H^2(\msvtwisted,\mathbb{Z})\hookrightarrow \widetilde{H}((S,\alpha),\mathbb{Z})\cong \bLambda_{24},$$
where $\widetilde{H}(S,\alpha,\mathbb{Z})\cong \bLambda_{24}$ is endowed with the Hodge structured described in \autoref{sec:twistedmoduli}.
Let $\sigma$ be the image of $E$ under the isometry \eqref{eq:hodgeisomtwisted}: then $\sigma$ is of $(1,1)$ type, has divisibility 3 in $\bL$ and $\sigma^2=-6$.

We want to prove that the Hodge embedding $\sigma^{\perp_{\bL}}\hookrightarrow \bLambda_{24}$ embeds in $\bLambda_{24}^{1,1}$ a copy of $\U(n)$ as a direct summand. 
To this end, observe that the composition
\begin{equation}\label{hodgeemb}
\sigma^{\perp}\rightarrow {E^{\perp}}\cong H^2(\msv,\mathbb{Z})\hookrightarrow \widetilde{H}((S,\alpha),\mathbb{Z})\cong \bLambda_{24}\cong \U^{\oplus 4}\oplus \E_8(-1)^{\oplus 2}
\end{equation}
defines a compatible Hodge embedding. 
Let $B\in H^2(S,\mathbb{Q})$ be such that $\mathrm{exp}(B)=\alpha$. Note that the vectors $n(1,B,B^2/2)$ and $(0,0,1)$ generate a copy of $\U(n)$ in $\widetilde{H}((S,\alpha),\mathbb{Z})^{1,1}\cong \bLambda_{24}^{1,1}$ as \eqref{hodgeemb} is a compatible Hodge embedding.

For the other direction we know that there exits \(\sigma \in \bL^{1,1}\) of square \(-6\) and divisibility \(3\) such that \(\sigma^{\perp_{\bL}}\hookrightarrow \bLambda_{24}\) embeds a copy of $\U(n)$ in $\bLambda_{24}^{1,1}$.
Let $e_n,f_n$ be the standard isotropic generators of $\U(n)$. Choosing $n$ minimal, we can assume that $e_n=\colon e$ is primitive in $\bLambda_{24}$. Hence we can complete $e$ to a copy of $\U=\langle e, f\rangle$ such that there is an orthogonal decomposition 
$$\bLambda_{24}\cong \U\oplus \mathbf{R}.$$
With respect to such decomposition the second basis vector $f_n$ of  $\U(n)$ can be written as 
$$ f_n=\gamma+nf+ke,$$
for some $\gamma\in \mathbf{R}$ and $k\in \mathbb{Z}$.\\
Similarly, a generator of the \((2,0)\)-part of the Hodge structure on $\bLambda_{24}$ induced by $\sigma^{\perp}$ is orthogonal to $e$ and hence is of the form $\omega+le$ with $\omega \in \mathbf{R}\otimes \mathbb{C}$ and $l\in \mathbb{Z}$. Moreover, since the same generator must be also orthogonal to $f_n$, we have that 
$$\gamma\cdot \omega+nl=0.$$
Setting $B=-\dfrac{\gamma}{n}$ we may write 
$$\omega+le=\omega+B\wedge \omega,$$
where $B\wedge \omega$ stands for $(B\cdot \omega)e$.
By the surjectivity of the period map, there exists a twisted K3 surface $(S,\alpha)$ with $\alpha=\mathrm{exp}(B)$ such that $H^2(S,\mathbb{Z})\cong \mathbf{R}$, identifying $H^{2,0}(S)=\langle\omega\rangle$.
Observe that the above construction defines an Hodge isometry between $\widetilde{H}(S,\alpha,\mathbb{Z})$ and $\bLambda_{24}$ with the Hodge structure induced by the Hodge embedding $\sigma^{\perp}\hookrightarrow \bLambda_{24}$.

We consider \(w \in \widetilde{H}(S,\alpha,\theta)\) such that under the above isometry \[[w]\cong (\sigma^{\perp_{\bL}})^{\perp_{\bLambda_{24}}}.\]
By construction \(w^2=2\) and either \(w\) or \(-w\) is a positive Mukai vector.
If we take \(v\colonequals 2w\) we can embed
\[H^{2}(\msvtwisted,\mathbb{Z}) \cong v^{\perp}\hookrightarrow \bLambda_{24}.\]
Again, note that $\msvtwisted$ admits a symplectic resolution $\tmsvtwisted$ whose exceptional divisor $E$ has square $-6$ and divisibility $3$ and 
$$[E]^{\perp}\cong H^2(\msvtwisted,\mathbb{Z}).$$
To construct the isometry between $H^2(X,\mathbb{Z})$ and $H^2(\tmsvtwisted,\mathbb{Z})$ we proceed as in \autoref{FGG}.
\end{proof}

\section{Birationality criteria for the Li-Pertusi-Zhao variety}\label{sec:biratLPZ}

We now give an application of \autoref{FGG} and \autoref{FGGtwisted} in the case of the $\og$ type variety defined by Li, Pertusi and Zhao associated with a smooth cubic fourfold \cite{LPZ}. We first recall briefly the construction and the main properties of this variety.

Given a smooth cubic fourfold $Y$ one defines the Kuznetsov component as 
\[
\cA_Y := \langle \cO_Y, \cO_Y(1), \cO_Y(2) \rangle^\perp \subset D^b(Y)
\]
which is a triangulated subcategory that shares several properties with the derived category of a K3 surface. Its topological $K$-theory
\[
\Ktop(\cA_Y) := \langle [\cO_Y],[\cO_Y(1)], [\cO_Y(2)] \rangle^\perp \subset \Ktop(Y)
\]
is a lattice endowed with the restriction of the Euler form on $\Ktop(Y)$. It has a natural pure Hodge structure of weight 2 obtained by pulling back the one from the cohomology of $Y$, see \cite{AT}.
For any smooth cubic fourfold the lattice $\Ktop(\cA_Y)$ contains two distinguished Hodge classes defined as $\lambda_i := \pr[\cO_L(i)] \text{ for $i=1,2$ }$,
where $L$ is a line on $Y$ and $\pr\colon \Ktop(Y) \to~\Ktop(\cA_Y)$ is the projection. In particular, by \cite[Proposition 2.3]{AT} we have
$$\langle \lambda_1,\lambda_2\rangle\cong \A_2$$
and the Mukai vector $v\colon \Ktop(\cA_Y) \to H^*(Y, \mathbb Q)$ restricts to a Hodge isometry
\[
\langle \lambda_1,\lambda_2\rangle^\perp \cong H^4_{\prim}  (Y, \mathbb Z).
\]
In \cite{LPZ} Li, Pertusi and Zhao showed that there exists a Bridgeland stability condition $\sigma$ for which the moduli space \[X_Y:=M_\sigma(2(\lambda_1+\lambda_2),\cA_Y)\] of semistable objects in $\cA_Y$ with fixed numerical invariant $2(\lambda_1+\lambda_2)\in \Ktop(\cA_Y)$ is a singular projective variety. Its singular locus consists of objects in $\Sym^2(M_\sigma(\lambda_1+\lambda_2,\cA_Y))$ which lie naturally inside the bigger moduli space. The \textit{Li-Pertusi-Zhao} (LPZ) variety $\tilde{X}_Y$ is constructed as a symplectic resolution of singularities of $X_Y$, obtained by blowing-up $X_Y$ along its  singular locus. In particular, $\widetilde{X}_Y$ is an ihs manifold of $\mathrm{OG}10$ type.
Crucial property for us is that we have a Hodge isometry 
\[
\Ktop(\cA_Y)\supseteq\langle \lambda_1+ \lambda_2\rangle^\perp \cong H^2 (X_Y, \mathbb Z)
\]
by \cite[Proposition~2.8]{GGO}.
Let $\bL:=H^2(\widetilde{X}_Y,\mathbb{Z})$  and let $\sigma\in \bL^{1,1}$ be the class of the exceptional divisor of the blow-up $\widetilde{X}_Y\rightarrow X_Y$. By \cite[\S 3]{LPZ}, we know that $\sigma^2=-6$ and $(\sigma,L)=3$.

In the following we denote by $\mathcal{C}$ the moduli space of smooth cubic fourfolds and by $\mathcal{C}_d$ the irreducible Hassett divisor consisting of special cubic fourfolds of discriminant $d$. These are cubic fourfolds which have a primitively embedded rank $2$ lattice $K\subset H^{2,2}(Y, \mathbb Z)$ containing the square of the hyperplane section, such that the Gram matrix of $K$ has discriminant $d$. Recall that \(\mathcal{C}_d\) is non-empty if and only if $d>6$ and $d\equiv 0,2 \pmod 6$, see \cite[Theorem 4.3.1]{Hassett_specialcubic}.
We consider the following two properties for $d$:
$$
\begin{array}{ll}
(**): &d \text{ divides }2n^2+2n+2 \text{ for some }n\in\mathbb{Z};\\
(**'): &\text{ in the prime factorization of } \dfrac{d}{2}\text{, primes }p \equiv 2 \ (3) \text{ appear with even exponents.}
\end{array}
$$

By \cite[Theorem 1.0.2]{Hassett_specialcubic} and \cite[page 1]{addington} condition $\mathrm{(**)}$ is equivalent to the existence of an associated K3 surface, i.e. a cubic fourfold $Y$ belongs to $\mathcal{C}_d$ with $d$ satisfying $\mathrm{(**)}$ if and only if there exists a polarised K3 surface $S$ with a Hodge isometry
\[
H^2_{\prim} (S, \mathbb Z) \simeq K^\perp \subset H^4(Y,\mathbb Z).
\]
Similarly, the condition $\mathrm{(**')}$ is equivalent to the existence of an associated twisted K3 surface, see \cite[Theorem 1.4]{Huybrechtscubic}.\\
Both conditions can be rephrased in purely lattice-theoretic terms. We follow \cite[Proposition 1.2]{huy_BiratGeom} and recall this description.

For any primitive vector $v\in \langle\lambda_1,\lambda_2\rangle^{\perp_{\Ktop}}$ (which is Hodge isometric to \(\A_2^{\perp_{\bLambda_{24}}}\))
we define
$$L_v\coloneqq(\A_2\oplus\mathbb{Z}v)^{sat}$$
to be the saturation of \(\A_2\oplus\mathbb{Z}v\) inside $\Ktop(\mathcal{A}_Y)$.
By \cite[\S 1.2]{huy_BiratGeom} one has
\begin{align}\label{starstar}
&d\text{ satisfies }\mathrm{(**)}\Leftrightarrow &\exists \ v_d\in \A_2^{\perp} \text{ such that }L_d:=L_{v_d} \text{ has discriminant } d \\ 
& &\text{ and }\U \hookrightarrow L_d \text{ primitively}, \qquad \qquad \qquad \qquad \qquad \quad \nonumber
\end{align}
\begin{align}\label{starstarprime}
&d\text{ satisfies }\mathrm{(**')}\Leftrightarrow &\exists \ v_d\in \A_2^{\perp} \text{ such that }L_d:=L_{v_d} \text{ has discriminant } d \\
& &\text{ and }\U(n) \hookrightarrow L_d \text{ for some }n\in \mathbb{N}. \qquad \qquad \qquad  \qquad\nonumber
\end{align}

\begin{theorem}\label{FGG_implies_GGO}
Let $\widetilde{X}_Y$ be an LPZ-variety associated to a smooth cubic fourfold $Y$. The following conditions are equivalent.
\begin{enumerate}[(i)]
\item The cubic fourfold $Y$ lies in the Hassett divisor $\mathcal{C}_d$ with $d$ satisfying $\mathrm{(**)}$.
\item The LPZ-variety $\widetilde{X}_Y$ is a numerical moduli space.
\end{enumerate}
\end{theorem}
\begin{proof}
We first show that $(i)\Rightarrow (ii)$. Given the resolution $\pi:\widetilde{X}_Y\rightarrow X_Y$ we denote by $\sigma$ the class of the exceptional divisor. The pullback $\pi^*$ embeds $H^2(X_Y,\mathbb{Z})$ in $H^2(\widetilde{X}_Y,\mathbb{Z})$. In particular
$$\pi^*H^2(X_Y,\mathbb{Z})=\sigma^{\perp}.$$
The above identification and the Hodge isometry in \cite[Proposition 2.8]{GGO} yield a Hodge embedding 
$$\sigma^{\perp}\xhookrightarrow{j} \Ktop(\mathcal{A}_Y)\cong \bLambda_{24}$$
such that 
$$j(\sigma^{\perp})=j(\pi^*H^2(X_Y,\mathbb{Z}))=\langle\lambda_1+\lambda_2\rangle^{\perp_{\Ktop}}.$$
Moreover we have a Hodge embedding $j':H^4(Y,\mathbb{Z})_{\mathrm{prim}}(-1)\rightarrow \Ktop(\mathcal{A}_Y),$
where the Hodge structure of $H^4(Y,\mathbb{Z})_{\mathrm{prim}}$ has a Tate twist decreasing weight by (-1,-1) and 
$$j'(H^4(Y,\mathbb{Z})_{\mathrm{prim}})=\langle\lambda_1,\lambda_2\rangle^{\perp}.$$
By hypothesis and \eqref{starstar}, there exists $v_d\in H^4(Y,\mathbb{Z})_{prim}$ of type (2,2) such that the saturation $L_d$ of $$\langle\lambda_1,\lambda_2,j'(v_d)\rangle\subset \Ktop(\mathcal{A}_Y)\cong \bLambda_{24}$$
contains a copy of $\U$ primitively. Moreover, note that by construction $L_d$ is contained in the (1,1)-part. 
The claim now follows by observing that, since $j$ and $j'$ are compatible Hodge embeddings, the Hodge structure on $\bLambda_{24}$ induced by $\sigma^{\perp}\xhookrightarrow{j}\bLambda_{24}$ is the same as the one induced by the isomorphism $\bLambda_{24}\cong\Ktop(\mathcal{A}_Y)$.

We now show $(ii)\Rightarrow (i)$. Let \(\widetilde{X}_Y\) be the LPZ variety associated with the cubic fourfold $Y$. We have the diagram
\begin{center}
\begin{tikzcd}
\langle \lambda_1,\lambda_2\rangle^{\perp_{\Ktop}}\cong H^4(Y,\mathbb{Z})_{\prim}(-1) \arrow[d, hook] \arrow[hook]{r}{j} & \Ktop(\mathcal{A}_Y) \\
\langle \lambda_1+\lambda_2\rangle^{\perp_{\Ktop}}\cong H^2(X_Y,\mathbb{Z})\cong \sigma^{\perp_{H^2(\widetilde{X}_Y)}} \arrow[hook]{r}{j'}                 & H^2(\widetilde{X}_Y,\mathbb{Z}),
\end{tikzcd}
\end{center}
which, in terms of abstract lattices, reads as 
\begin{center}
\begin{tikzcd}
\langle \lambda_1,\lambda_2\rangle^{\perp_{\bLambda_{24}}}\cong \U^{\oplus2}\oplus \E_8(-1)^{\oplus 2 }\oplus \A_2(-1) \arrow[d, hook] \arrow[r, hook] & \bLambda_{24}\\
\langle \lambda_1+\lambda_2\rangle^{\perp_{\bLambda_{24}}}\cong \U^{\oplus 3}\oplus \E_8(-1)^{\oplus 2 }\oplus [-2] \arrow[r, hook]                 & \U^{\oplus 3} \oplus \E_8(-1)^{\oplus 2} \oplus \A_2(-1) .
\end{tikzcd}
\end{center}
By assumption, the lattice \(\bLambda_{24}^{1,1}\) contains a copy of \(\U\) as a direct summand. Since the orthogonal complement of \(\sigma^{\perp_{H^{2}(\widetilde{X}_Y)}} \subset \bLambda_{24}\) is generated by a positive class of \((1,1)\) type, we have that
$$\rk((\sigma^{\perp})^{1,1})\geq 1,$$
as $\sigma^{\perp}$ must contain a negative class of $\U$ inside $\bLambda_{24}^{1,1}$.
Moreover \(\sigma^{\perp}=\langle \lambda_1+\lambda_2\rangle^{\perp}\) contains the (1,1) class \(\lambda_1 - \lambda_2\) which is of positive square, hence 
\[\rk((\sigma^{\perp})^{1,1})\geq 2.\]
We deduce that \( \rho(X_Y) \geq 2\) and by the diagrams above \[\rho(Y)_{\prim}:=\rk (H^{2,2}(Y,\mathbb{Z})_{\prim}) \geq 1,\] which means that the cubic fourfold \(Y\) belongs to a Hassett divisor \(\mathcal{C}_d\) for some \(d \in \mathbb{N}\). 

We now show that \(d\) satisfies the \(\mathrm{(**)}\) condition, i.e. that there exists a primitive vector \(v_d\in H^{2,2}(Y,\mathbb{Z})_{\prim}\) associated to \(d\) such that \(L_d:=\langle \lambda_1, \lambda_2, v_d \rangle^{sat}\) contains primitively a copy of \(\U\). 

Note that since \(\left ( \langle \lambda_1,\lambda_2 \rangle^{\perp_{\bLambda_{24}}} \right )^{1,1} \) is equal to  \(H^{2,2}(Y, \mathbb{Z})_{\prim}(-1)\), which is negative definite, then the lattice \(\left ( \langle \lambda_1,\lambda_2 \rangle^{\perp_{\bLambda_{24}}} \right )^{1,1} \cap \U\) is generated by a negative class that we call \(v_d\).

We now need to show that $L_d$ contains a copy of $\U$ primitively. To do so, observe that the copy of $\U$ in $\bLambda_{24}^{1,1}$ intersects the lattice $\A_2$ generated by $\lambda_1$ and $\lambda_2$.
In fact 
$\bLambda_{24}^{1,1}\supseteq \langle\lambda_1,\lambda_2\rangle\oplus H^{2,2}(Y,\mathbb{Z})_{\prim}\cong \A_2\oplus H^{2,2}(Y,\mathbb{Z})_{\prim}.$
Since $H^{2,2}(Y,\mathbb{Z})_{\prim}$ is negative definite, it follows that $\U\cap \A_2$ has rank 1 and it is generated by a positive element. 
Since $\U$ is unimodular, the embedding $\U\hookrightarrow L_d$ is primitive.
\end{proof}

Now the result \cite[Theorem 3.2]{GGO} by F. Giovenzana, L. Giovenzana and Onorati is a straightforward consequence of \autoref{FGG_implies_GGO} and \autoref{FGG}.
\begin{corollary}\label{corollary} Keep the notation as above. The following conditions are equivalent.
\begin{enumerate}[(i)]
\item The cubic fourfold $Y$ lies in the Hassett divisor $\mathcal{C}_d$ with $d$ satisfying $\mathrm{(**)}$.
\item There exists a K3 surface $S$, a non-primitive Mukai vector $v=2w$  and a $v-$generic polarization $\theta$ on $S$ such that $\widetilde{X}_Y$ is birational to $\tmsv$.
\end{enumerate}
\end{corollary}

\begin{theorem}\label{FGGtwisted_implies_GGOtwisted}
Let $\widetilde{X}_Y$ be an LPZ-variety associated to a smooth cubic fourfold $Y$. The following conditions are equivalent.
\begin{enumerate}[(i)]
\item The cubic fourfold $Y$ lies in the Hassett divisor $\mathcal{C}_d$ with $d$ satisfying $\mathrm{(**')}$.
\item The LPZ-variety $\widetilde{X}_Y$ is a twisted numerical moduli space.
\end{enumerate}
\end{theorem}
\begin{proof}
The proof is identical to that of \autoref{FGG_implies_GGO} provided that one replaces \(\U\) with \(\U(n)\).
\end{proof}

We are now able to state a result analogous to \autoref{corollary} in the twisted case, which was proved in \cite[Proposition~3.1.(2)]{GGO} only under the additional assumption that the map between $X_Y$ and $\msvtwisted$ preserves the stratification of the singular locus . 
Our criterion allows to circumvent this issue by looking just at the lattice theoretic setting.

\begin{corollary} Keep the notation as above. The following conditions are equivalent.
\begin{enumerate}[(i)]
\item The cubic fourfold $Y$ lies in the Hassett divisor $\mathcal{C}_d$ with $d$ satisfying $\mathrm{(**')}$.
\item There exists a twisted K3 surface $(S,\alpha)$, a non-primitive Mukai vector $v=2w$  and a $v-$generic polarization $\theta$ on $S$ such that $\widetilde{X}_Y$ is birational to $\tmsvtwisted$.
\end{enumerate}
\end{corollary}

\section{Numerically induced birational transformations}\label{sec:numind}
As a consequence of \autoref{FGG}, when $X$ is a numerical moduli space there exist a K3 surface $S$, a Mukai vector $v$ and a polarization $\theta$ such that $X$ is birational to $\tmsv$. We now want to investigate which birational transformations of $X$ come from automorphisms of the K3 surface $S$.
To this end, we shall give two preliminary definitions. Then in \autoref{sub:ind_involut} we consider symplectic birational involutions and determine the induced ones. 

\begin{definition}\label{definduced}
Let $(X,\eta)$ be a marked pair of $\og$ type and let be $G\subset \mathrm{Bir}(X)$ be a finite subgroup of birational transformations of $X$. 
The group $G$ is called an \textit{induced group of birational transformations} if there exists a $K3$ surface $S$ with an injective group homomorphism $i\colon G\hookrightarrow \mathrm{Aut}(S)$, a $G$-invariant Mukai vector $v\in \widetilde{H}(S,\mathbb{Z})^G$ and a $v$-generic polarization $\theta$ on \(S\) such that the action induced by $G$ on $\tmsv$ via $i$ coincides with given action of $G$ on $X$.
\end{definition}

\begin{definition}\label{defnumericallyinduced}
Let $X$ be a smooth irreducible holomorphic symplectic manifold of $\og$ type and let $\eta\colon H^2(X,\mathbb{Z})\rightarrow \bL$ be a fixed marking. 
A finite group $G\subset \mathrm{Bir}(X)$ of birational transformations is called a \textit{numerically induced group of birational transformations} if there exists a class
$\sigma\in \mathrm{NS}(X)$ such that $\sigma^2=-6$ and $(\sigma,\bL)=3$ such that:
\begin{itemize}
    \item[(i)] $\sigma$ is $G$-invariant;
    \item[(ii)] Given the Hodge embedding $\sigma^{\perp}\hookrightarrow \bLambda_{24}$, the induced action of $G$ on $\bLambda_{24}$ is such that the $(1,1)$ part of the invariant lattice $(\bLambda_{24})^{G}$ contains $\U$ as a direct summand.
\end{itemize}
\end{definition}

\begin{remark} Observe that if \(X\) admits a finite subgroup $G\subseteq \Bir(X)$ which is numerically induced, then $X$ is automatically a numerical moduli space.
\end{remark}

In what follows we give all statements assuming $G\subseteq \mathrm{Bir}(X)$. If the statements hold true for $G\subseteq \Aut(X)$, then $G$ will be called an \textit{(numerically) induced group of automorphisms}.

\begin{proposition}\label{autgroup}
If \(\varphi \in \Aut(S)\) is an automorphism of the K3 surface \(S\), the vector \(v=2w \in \widetilde{H}(S,\mathbb{Z})\) with \(w^2=2\)
is a \(\varphi\)-invariant Mukai vector on \(S\), and \(\theta\) is a \(v\)-generic and \(\varphi\)-invariant polarization on \(S\), then \(\varphi\)
induces an automorphism \(\widetilde{\varphi}\) on \(\tmsv\). Moreover this automorphism is numerically induced.
\end{proposition}
\proof
The definition of stability implies that if a coherent sheaf  \(\mathscr{F}\) is \(\theta\)-stable then \(\varphi^{*}\mathscr{F}\) is \(\theta\)-stable, see \cite[Proposition 1.32]{MW15}. As any \(\theta\)-semistable sheaf is an iterated extension of \(\theta\)-stable sheaves, we get that $\varphi^*$ preserves \(\theta\)-semistability. Hence we have a regular automorphism \(\hat{\varphi}\) of $\msv$. The singular locus of \(\msv\) is preserved by \(\hat{\varphi}\) and so we get a lift \(\widetilde{\varphi}\) of \(\hat{\varphi}\) on the blow-up \(\tmsv\). We call \(\sigma\) the class of the exceptional divisor which is fixed by the induced action. The primitive embedding 
\( \sigma^{\perp} = H^{2}(\msv, \mathbb{Z}) \hookrightarrow \widetilde{H}(S,\mathbb{Z}) \cong \bLambda_{24}\)
induces an action of \(\widetilde{\varphi}\) on \(\sigma^{\perp}\), which in turn yields an action on \(\widetilde{H}(S,\mathbb{Z}) \cong \bLambda_{24}\). This action coincides with that induced by \(\varphi\) on \(\widetilde{H}(S,\mathbb{Z})\). Observe that the sublattice \(H^0
(S,\mathbb{Z}) \oplus H^4
(S,\mathbb{Z})\cong \U\) is preserved by the action of \(\widetilde\varphi\) and it is contained in the \((1, 1)\) part
of the lattice \(\bLambda_{24}\), so \(\widetilde{\varphi}\) is numerically induced. 
\endproof

\begin{remark}
In the assumption of \autoref{autgroup}, we ask that the polarization $\theta\in \NS(S)$ is $\varphi$-invariant, so that the automorphism of the K3 surface induces an automorphism of $\tmsv$. Note that this condition is never verified for a nontrivial symplectic automorphism of $S$ according to \cite[Theorem 1.1]{GGOV}. On the other hand the case of induced nonsymplectic involutions is being investigated in a companion paper \cite{BG}.
 
\end{remark}

In the above setting, if the polarization $\theta$ is not $\varphi$-invariant, we may still prove that we get at least a birational self map of $\tmsv$.
To this end we need a generalized notion of stability: in fact, so far we have dealt with moduli spaces of sheaves on K3 surfaces with respect to Gieseker stability. We now consider moduli spaces of Bridgeland semistable objects, and recall some results about their birational geometry.

Given a projective K3 surface \(S\), Bridgeland constructed stability conditions on $D^b(S)$ and showed that they form a complex manifold denoted by $\Stab(S)$.
Moreover, for any Mukai vector $v\in \widetilde{H}(S,\mathbb{Z})$ there exists a Bridgeland stability condition $\tau$ for which Gieseker and $\tau$-semistable objects with Mukai vector $v$ coincide. This singles out a distinguished connected component $\Stab^\dagger(S)$ of $\Stab(S)$ that has a wall and chamber structure, see \cite[Sec.14]{Tom}. Any stability condition in the same chamber determines the same stable objects, thus we will call \emph{Gieseker chamber} the one that recovers Gieseker stability.

Let \(v=2v_0\) be a Mukai vector of O'Grady type and let \(\tau \in \Stab^\dagger(S)\) be a \(v\)-generic stability condition. The moduli space \(M_{v}(S,\tau)\) of \(\tau\)-semistable objects with Mukai vector $v$ is a singular projective symplectic variety admitting a symplectic resolution of singularities \(\widetilde{M}_v(S,\tau)\) of $\og$ type. For any such moduli space the natural stratification of the singular locus is given by
\[
M_{v_0}(S,\tau) \subset \Sym^2 (M_{v_0}(S,\tau)) \subset M_{v}(S,\tau).
\]
We recall the following result.
\begin{proposition}[{\cite[Theorem 5.4]{Meachan_Zhang}}]\label{Meachan-Zhang}
In the above setting, if $v\in \widetilde H(S,\mathbb Z)$ is a Mukai vector of O'Grady type and $\tau, \tau'\in \Stab^\dagger(S)$ are \(v\)-generic stability conditions, then there exists a birational map $\psi\colon M_{v}(S,\tau)\dashrightarrow M_{v}(S,\tau')$, which is stratum preserving in the sense that $\psi$ is defined on the generic point of each stratum of the singular locus of $M_{v}(S,\tau)$ and maps it to the generic point of the $M_{v}(S,\tau')$.
\end{proposition}
We are now in a position to prove our assertion.
\begin{proposition}\label{autbir}
If \(\varphi \in \Aut(S)\) is an automorphism of the K3 surface \(S\), \(v \in \widetilde{H}(S,\mathbb{Z})\)
is a \(\varphi\)-invariant Mukai vector on \(S\), and \(\theta\) is a polarization on \(S\), then \(\varphi\)
induces a birational transformation \(\widetilde\varphi\) on \(\tmsv\). Moreover this birational transformation is numerically induced.
\end{proposition}
\begin{proof}
Let $\tau$ be a Bridgeland stability condition in the Gieseker chamber: then any \(\theta\)-stable sheaf \(\mathscr{F}\) is $\tau$-stable. As chambers are open, for a generic $\tau$ the stability condition $\varphi^*(\tau)$ does not lie on any wall for the Mukai vector of $\varphi^*\mathscr{F}$ and \(\varphi^{*}\mathscr{F}\) is \(\varphi^*(\tau)\)-stable. For this reason we have an isomorphism  \({\widehat{\varphi}\colon M_{v}(S,\tau)\xrightarrow{\sim} M_{v}(S,\varphi^*(\tau))}\) which restricts to an isomorphism between the singular loci. Composing $\widehat{\varphi}$ with the birational stratum preserving transformation ${M_{v} (S,\varphi^*(\tau))\xrightarrow{\sim} M_{v}(S,\tau)}$ of \autoref{Meachan-Zhang}, we get a birational self-transformation $\widetilde\varphi$ of $M_{v}(S,\tau)$. As $\widetilde \varphi$ preserves the singular locus, it lifts to a birational self-transformation of the smooth irreducible holomorphic symplectic manifold ${\widetilde M_{v}(S,\tau)=\tmsv}$. With a slight abuse of notation we still denote it by \(\tilde\varphi\). By construction, this map preserves the exceptional divisor and its class \(\sigma \in H^2 (\tmsv,\mathbb{Z}\). As in \autoref{autgroup}, the primitive embedding 
\( \sigma^{\perp} \cong H^{2}(\msv, \mathbb{Z}) \hookrightarrow \widetilde{H}(S,\mathbb{Z}) \cong \bLambda_{24}\) and
the induced action of \(\widetilde{\varphi}\) on \(\sigma^{\perp}\) give an action on \(\widetilde{H}(S,\mathbb{Z}) \cong \bLambda_{24}\) which coincides with that induced by \(\varphi\) on \(\widetilde{H}(S,\mathbb{Z})\). Since the sublattice \(H^0
(S,\mathbb{Z}) \oplus H^4
(S,\mathbb{Z})\cong \U\) is preserved by the action of \(\widetilde\varphi\) and it is contained in the \((1, 1)\) part
of the lattice \(\bLambda_{24}\), we conclude that $\widetilde{\varphi}$ is numerically induced. 
\end{proof}

\begin{theorem}\label{induced_aut}
Let $(X,\eta)$ be a smooth marked pair of $\og$ type and let $G\subset \mathrm{Bir}(X)$ be a finite subgroup. If $G$ is a numerically induced group of birational transformations, then there exists a K3 surface $S$ with an injective group homomorphism $G\hookrightarrow \mathrm{Aut}(S)$, a $G$-invariant Mukai vector $v$ and a $v$-generic polarization $\theta$ on $S$ such that $X$ is birational to $\tmsv$ and $G\subset\mathrm{Bir}(X)$ is an induced group of birational transformations. 
\end{theorem}

\begin{proof}
We first consider the case when $G$ is symplectic. Then the transcendental lattice $\mathbf{T}(X)$ is contained inside $\bL^G$. Since $G$ is numerically induced, there exists a class $\sigma\in \mathrm{NS}(X)$ with $\sigma^2=-6$ and $(\sigma,\bL)=3$ which is $G$ invariant. 
We consider the Hodge embedding $\sigma^{\perp_{\bL}}\hookrightarrow \bLambda_{24}$ and we call $w$ the generator of the orthogonal complement of the image of $\sigma^{\perp_{\bL}}$ inside $\bLambda_{24}$.
We can extend the action of \(G\) on \(\bLambda_{24}\) in such a way that the action on the orthogonal complement of $\sigma^{\perp_{\bL}}$ in \(\bLambda_{24}\) is trivial.
Note that, by construction, $w$ is of \((1,1)\) type and it is fixed by the induced action of \(G\) on $\bLambda_{24}$. Moreover by assumption we have that 
$$(\bLambda_{24})^G=\U \oplus \boldsymbol{\mathrm{T}}$$
with $\U$ contained in the \((1,1)\) part of the Hodge structure of $(\bLambda_{24})^G$.
Since \(\sigma\) is \(G\)-invariant then there is a Hodge embedding \(\bL_{G}=(\sigma^{\perp_{\bL}})=(\bLambda_{24})_{G}\), then \(\bL_{G} \hookrightarrow \bLambda_{24}\) and \(\bL_{G}^{\perp_{\bLambda_{24}}}= \bLambda_{24}^G=\U \oplus \mathbf{T}\).

Hence the embedding \(\bL_G \hookrightarrow \bLambda_{24}\) can be decomposed in \(\bL_G \xhookrightarrow{\varphi} \U^{\oplus3} \oplus \E_8(-1)^{\oplus 2} \hookrightarrow \bLambda_{24}\), where \(\varphi\) is a compatible Hodge embedding and by construction \(\bL_{G}^{\perp_{\U^{\oplus3} \oplus \E_8(-1)^{\oplus2}}} =\mathbf{T}\). Since \(\bL_G\) is of \((1,1)\)-type then \(\varphi(\bL_G) \subseteq (\U^{\oplus3} \oplus \E_8(-1)^{\oplus2})^{1,1}\). The lattice \(\U^{\oplus3} \oplus \E_8(-1)^{\oplus2}\) has an Hodge structure hence, by Torelli theorem for K3 surfaces, there exists a K3 surface \(S\) such that \(H^{2}(S,\mathbb{Z})\) is Hodge isometric to \(\U^{\oplus3} \oplus \E_8(-1)^{\oplus 2}\).
Moreover the action induced by \(G\) on \(H^{2}(S,\mathbb{Z})\) is a Hodge isometry since by construction \(H^{2}(S,\mathbb{Z})_G=\bL_G\) and \(H^{2}(S,\mathbb{Z})^G=\mathbf{T}\). Furthermore the induced action of \(G\) on \(H^{2}(S,\mathbb{Z})\) is a monodromy operator due to the maximality of the monodromy group of a K3 surface in the group of isometries, and the coinvariant lattice \(H^{2}(S, \mathbb{Z})_G\) does not contain wall divisors for the K3 surface, i.e. vectors of square \(-2\). In fact the coinvariant lattice \(H^{2}(S,\mathbb{Z})_G\) is isometric to the coinvariant lattice \(\bL_G\) with respect to the induced action of \(G\subset \Bir(X)\) on the second integral cohomology lattice of the manifold of \(\og\) type, hence it does not contain prime exceptional divisors by \cite[Theorem 2.2]{GGOV}.  The characterization of prime exceptional divisors for manifolds of \(\og\) type (see \cite[Proposition 3.1]{MongardiOnorati22}) thus implies that \(H^{2}(S,\mathbb{Z})_G\) does not contain any vectors of square \(-2\). We apply \cite[Theorem 1.3]{markman2011survey} and we conclude that \(G \subseteq \Aut(S)\). Since \(w\) is invariant by construction, \autoref{FGG} implies that \(X\) is birational to \(\tmsv\), where \(v=2w\) is an algebraic Mukai vector on \(S\) and \(v^{2}=(2w)^{2}=8\). Moreover the induced action of \(G \subseteq \Aut(S)\) on \(H^{2}(\tmsv, \mathbb{Z})\) coincides by construction with the given action of \(G\) on \(H^{2}(X,\mathbb{Z})=\bL\).  
We deduce that \(G\subset\Bir(X)\) is an induced group of birational transformations.

Assume now that every non-trivial element of $G$ has a non-symplectic action. Without loss of generality, choosing $X$ generic, we can assume that $\bL^{G}=\NS(X)$. As before we have that $(\bLambda_{24})^{G}= \U \oplus \boldsymbol{\mathrm{T}}$ and we can consider the K3 surface \(S\) associated to the Hodge structure induced on $\U^{\perp} \subset \bLambda_{24}$: indeed, by the genericity of $X$,  we can assume that \( (\bLambda_{24})^{1,1}=(\bLambda_{24})^{G}\) so that \((\U^{\perp})^{1,1}\) coincides with \(\boldsymbol{\mathrm{T}}\), which has thus signature (1,-) and is of type (1,1).
Again by \autoref{FGG} $X$ is birational to a moduli space of sheaves on $S$. The group $G$ is a group of Hodge isometries of $S$ preserving $\boldsymbol{\mathrm{T}}=\NS(S)$: in particular it preserves all K\"ahler classes and we can conclude as in the symplectic case.


Suppose now that $G$ is any numerically induced group of birational transformations and let \(G_s\) be its symplectic part. Considering the action of \(G_s\) we obtain a K3 surface $S$ as in the first step with \(G_s\subset \Aut(S)\).  We can extend this action to an action of $G$ by applying the second step to the quotient group $\widehat{G}\coloneqq G/G_s$.
\end{proof}

\subsection{Induced involutions}\label{sub:ind_involut}
In this section we give an application of the criterion about induced automorphisms given in  \autoref{induced_aut}. The lattice-theoretic criterion that we prove allows us to determine if a birational transformation of a manifold of \(\og\) type which is at least
birational to \(\tmsv\), is induced by an automorphism of the K3 surface \(S\). Let us be more definite.  By \cite{GOV20} we know that there are no regular symplectic involutions on a manifold of \(\og\) type. However, \cite[Theorem 1.1]{Marquand:sympl.birat.} provides a lattice-theoretic classification of birational symplectic involutions. In the paper, the authors consider a manifold of \(\og\) type,
a fixed marking \(\eta \colon H^2(X,\mathbb{Z}) \to \bL\) of \(X\), and classify invariant and coinvariant
sublattices, denoted by \(H^{2}(X,\mathbb{Z})_+ \cong \bL^G\) and \(H^{2}(X,\mathbb{Z})_- \cong \bL_G\) respectively. In particular, they classify the images of symplectic birational involutions with respect to the representation
map
\begin{align*}
\eta_{*} \colon \Bir(X) &\to O(\bL)\\
f &\mapsto \eta \circ f^{*} \circ \eta^{-1}.
\end{align*}
We prove that there exists a unique induced birational symplectic involution on a manifold of \(\og\) type and we give the following characterization in terms of its lattice-theoretic behaviour. 

\begin{remark}\label{remark trivial action}
If the action is induced then the action on the discriminant group is trivial. Indeed it is enough to note that a generator of the discriminant group \(A_{\bL}\) is the class \([\frac{\sigma}{3}]\) where \(\sigma\) is the class of the exceptional divisor.
\end{remark}
\begin{theorem}\label{example_ind_aut}
Let \((X, \eta)\) be a marked pair of \(\og\) type. Assume that \(X\) is a numerical moduli space and let \(\varphi \in \Bir(X)\) be a birational symplectic involution. Then \(\varphi\) is induced if and only if
\[H^{2}(X,\mathbb{Z})^{\varphi} \cong \U^{\oplus{3}} \oplus \E_8(-2) \oplus \A_2(-1)  \ \text{and} \ H^{2}(X,\mathbb{Z})_{\varphi} \cong \E_8(-2).\]
\end{theorem}

\proof
If \(\varphi\) is an induced symplectic birational transformation, then by \autoref{autbir} we have that \(\varphi\) is numerically induced. In \cite[Theorem 1.1]{Marquand:sympl.birat.} we have a classification of symplectic birational involutions on manifolds of \(\og\) type. As the action of an induced automorphism on the discriminant is trivial, we deduce that the $H^2(X,\mathbb Z)^\varphi$ is either $\U^{\oplus{3}} \oplus \E_8(-2) \oplus \A_2(-1)$ or $\E_6(-2)\oplus \U(2)^{\oplus2}\oplus[2]\oplus [-2]$. \autoref{induced_aut} excludes the latter: in fact, if the action is induced, then it is numerically induced. In particular there exists a primitive invariant class $\sigma$ of square $-6$ and divisibility $3$, in contradiction with the fact that any element in $\E_6(-2)\oplus \U(2)^{\oplus2}\oplus[2]\oplus [-2]$ has even divisibility.

To prove the converse, let $\varphi$ be a birational symplectic involution with \[H^2(X,\mathbb Z)^\varphi= \U^{\oplus{3}} \oplus \E_8(-2) \oplus \A_2(-1)\text{ and } H^{2}(X,\mathbb{Z})_{\varphi} \cong \E_8(-2).\]
We first prove that $\varphi$ is numerically induced.
Since \(X\) is a numerical moduli space then there exists a class $\sigma\in \NS(X)$ with \(\sigma^2=-6\) and $(\sigma,\bL)=3$ such that, given the Hodge embedding $\sigma^{\perp}\hookrightarrow \bLambda_{24}$, then $\bLambda_{24}^{1,1}$ contains $\U$ as a direct summand. Observe that such copy of $\U$ must be contained in \(H^{2}(X,\mathbb{Z})^{\varphi}\) as it is not possible to realize $\U$ as a sublattice of $\E_8(-2)$.
Moreover, since $H^{2}(X,\mathbb{Z})_{\varphi} \cong\E_8(-2)$ does not contain any divisibility 3 class, then $\sigma$ must lie in \(H^2(X,\mathbb Z)^\varphi\) i.e. it is invariant and $\varphi$ is thus numerically induced.
The result now follows from \autoref{induced_aut}.
\endproof

\bibliographystyle{amsplain}
\bibliography{main}
\end{document}